\documentclass[10pt]{article}
\usepackage[margin=1in]{geometry}
\usepackage{amsthm, amsmath, amsfonts, amssymb}
\usepackage{graphicx}
\usepackage{authblk}
\usepackage{tikz}
\usepackage{listings}
\usepackage{url}
\usepackage{xcolor}
\usepackage{float}
\usepackage{ifthen}
\usepackage{hyperref}
\usetikzlibrary{arrows.meta, calc, automata, positioning}

\definecolor{codegray}{gray}{0.97}
\definecolor{codeblue}{rgb}{0.1,0.1,0.6}
\definecolor{codegreen}{rgb}{0.0,0.5,0.0}
\definecolor{codepurple}{rgb}{0.5,0.0,0.5}
\definecolor{codered}{rgb}{0.6,0.1,0.1}
\definecolor{codeorange}{rgb}{0.8,0.3,0.0}

\DeclareMathOperator{\constantterm}{ct}
\newcommand{\ct}[1]{\constantterm\left[#1\right]}

\theoremstyle{plain}
\newtheorem{theorem}{Theorem}[section]
\newtheorem{conjecture}[theorem]{Conjecture}
\newtheorem{proposition}[theorem]{Proposition}

\theoremstyle{definition}
\newtheorem{definition}[theorem]{Definition}

\title{Automatic Bounds on Constant Term Sequences Modulo Primes}
\author{Justin Offutt\thanks{Undergraduate student researcher. This work was conducted as part of the LEMMA research group at Indiana University.}}
\affil{Department of Mathematics, Indiana University\\Bloomington, IN, USA\\jcoffutt@iu.edu}
\date{\today}

\begin{document}
	
	\maketitle
	
	\begin{abstract}
		This paper provides counterexamples to a previously conjectured upper bound on the first index \( n_0 \) at which a zero appears in constant term sequences of the form \( A_p(n) := \operatorname{ct}(P^n) \bmod p \), where \( P(t) \in \mathbb{Z}[t, t^{-1}] \). The conjecture posited that the first zero must occur at some index \( n_0 < p^{\deg(P)} \) for univariate Laurent polynomials. We prove an automaton state-based bound for univariate polynomials \( n_0 < p^{\kappa(P, p)} \), where \( \kappa(P, p) \) is the automaticity of \( (A_p(n))_{n \geq 0} \) over \( \mathbb{F}_p \). We support our theoretical results with randomized experiments on low-degree Laurent polynomials, and propose the \( \kappa(P, p) \)-based bound as a practical alternative to the general worst-case bound \( B_{P,p} = p^{p^{(2\deg(P)-1)^r}} \) arising from the Rowland–Zeilberger construction.
	\end{abstract}
	
	\section{Introduction}
	Constant term sequences, defined by taking the constant term of powers of Laurent polynomials, generally display intricate algebraic and modular structure. These sequences can often contain combinatorial information and arise naturally in areas such as analytic combinatorics and number theory. Classical results, such as Christol's theorem \cite{christol} and the work of Dwork \cite{dwork}, reveal deep connections between algebraic functions and their reductions modulo primes, hinting that constant term sequences mod \( p \) may inherit automata-theoretic structure. Motivated by these connections, we study sequences of the form
	\[
	A_p(n) := \operatorname{ct}(P^n) \mod p,
	\]
	where \( P(t) \in \mathbb{Z}[t, t^{-1}] \) is a Laurent polynomial. Our focus is on understanding when these sequences first attain the value zero modulo \( p \), and how the algebraic and automata-theoretic properties of \( P \) influence this behavior.
	
	The following empirical analysis on these sequences leverages a Sage library developed by Nadav Kohen, which provides tools for manipulating and computing constant term sequences. The source code is publicly available at \texttt{\url{https://github.com/nkohen/ConstantTermSequences}}, and the implementation developed specifically for our experiments is available at \texttt{\url{https://github.com/jcbopit/CTSExperiments/}}. 
	
	We now outline the structure of the paper. In Section 2, we introduce basic definitions and provide a motivating example. Section 3 describes our experimental methodology. Section 4 presents counterexamples to the conjectured bound. Section 5 proposes an alternative automaton-based bound. We conclude with a discussion of implications and directions for future work.

	\section{Background and Definitions}
	
	\begin{definition}
		Let \( P(x), Q(x) \) be multivariate Laurent polynomials in variables \( x = (x_1, x_2, \dots, x_m) \). The \emph{constant term sequence} \( A(n) \) is the sequence indexed by \( n \in \mathbb{N} \), defined by
		\[
		A(n) := \text{ct}\left[P(x)^n Q(x)\right],
		\]
		where \( \text{ct}[\cdot] \) denotes the \emph{constant term} operator which extracts the coefficient of the monomial \( x_1^0 x_2^0 \cdots x_m^0 \) in the expansion of \( P(x)^n Q(x) \). This defines an infinite sequence:
		\[
		(A(n))_{n \geq 0} = A(0), A(1), A(2), \dots.
		\]
	\end{definition}
	
	\subsection{A Motivating Example}
	
	An example of a constant term sequence arises from the Laurent polynomial \( P(x, y) = x^{-1} + x + y^{-1} + y \). The infinite sequence \( (A(n))_{n \geq 0} \) defined by:
	\[
	A(n) = \text{ct}\left[(x^{-1} + x + y^{-1} + y)^n\right] = 1,\;0,\;4,\;0,\;36,\;0,\;400,\;0,\;4900,\; \dots
	\]
	counts the number of two-dimensional lattice paths of length \( n \) composed of steps in the four cardinal directions: left (\( L = x^{-1} \)), right (\( R = x \)), down (\( D = y^{-1} \)), and up (\( U = y \)), that begin and end at the origin. This is only possible when the number of left and right steps are equal (L = R), and the number of up and down steps are equal (U = D), which explains why the constant term, the number of ways to return to the origin, is zero for odd indices and nonzero for even indices. This sequence is catalogued in the OEIS as entry \href{https://oeis.org/A005568}{A005568}.

	For example, suppose we want to know how many paths return to the origin in exactly \( n = 6 \) steps:

	\begin{align*}
		A(n) &= \text{ct}\left[(x^{-1} + x + y^{-1} + y)^n\right] = 1,\;0,\;4,\;0,36,\;0,\;\textbf{400},\; \dots \\
		A(6) &= \textbf{400}
	\end{align*}
	
	The $6$-th index of the constant term sequence $A(n)$ encodes that there are \emph{400} ways to leave the origin and return to it using 6 steps in the cardinal directions!
	
	Below is a diagram of a few such lattice paths of length \( n = 6 \).

\begin{center}
	\begin{tikzpicture}[scale=1, line width=1.5pt, >=Stealth]

		\draw[step=1cm, gray!30, thin] (-3.5,-3.5) grid (3.5,3.5);

		\draw[->, thick] (-3.6,0) -- (3.8,0) node[right] {\footnotesize \(x\)};
		\draw[->, thick] (0,-3.6) -- (0,3.8) node[above] {\footnotesize \(y\)};

		\node[circle,fill=black,inner sep=2pt] at (0,0) {};
		\node[below right] at (0,0) {\footnotesize (0,0)};
		
		\draw[->, red, shorten >=2pt, shorten <=2pt]
		(0,0) -- (1,0) -- (1,1) -- (2,1) -- (2,0) -- (1,0) -- (0,0);
		\draw[->, blue, shorten >=2pt, shorten <=2pt]
		(0,0) -- (0,1) -- (0,2) -- (-1,2) -- (-1,1) -- (-1,0) -- (0,0);
		\draw[->, green, shorten >=2pt, shorten <=2pt]
		(0,0) -- (0,-1) -- (-1,-1) -- (-1,-2) -- (0,-2) -- (0,-1) -- (0,0);

		\begin{scope}[shift={(4.5,2.5)}]
			\draw[gray!50, rounded corners] (-0.3,-2) rectangle (4.35,0.5);
			\node[anchor=west, red]   at (0, 0) {\rule{10pt}{1.5pt} \quad R, U, R, D, L, L};
			\node[anchor=west, blue]  at (0, -0.6) {\rule{10pt}{1.5pt} \quad U, U, L, D, D, R};
			\node[anchor=west, green] at (0, -1.4) {\rule{10pt}{1.5pt} \quad D, L, D, R, U, U};
			\node at (1.75, 0.75) {\footnotesize \textbf{Path Legend}};
		\end{scope}
		
	\end{tikzpicture}
\end{center}

\subsection{Automata and Base-\( p \) Representation of Constant Term Sequences}

Constant term sequences modulo a prime \( p \) are closely connected to automata theory via Christol's theorem \cite{christol}. This classical result states that a sequence \( a(n) \) over \( \mathbb{F}_p \) is \( p \)-automatic if and only if its generating function is algebraic over \( \mathbb{F}_p(x) \).

Since constant term sequences of the form
\[
A_p(n) = \text{ct}[P^n] \mod p
\]
have rational generating functions, and hence are algebraic, Christol’s theorem guarantees that they are \( p \)-automatic. That is, there exists a finite automaton which, given the base-\( p \) expansion of \( n \) (read from least to most significant digit), computes \( A_p(n) \).

\medskip

Concretely, a finite automaton processes the digits of \( n \) one by one, transitioning between a finite set of states based on each digit, and then outputs the value of \( A_p(n) \) according to the final state reached.

\medskip

To illustrate this, we present an automaton over \( \mathbb{F}_2 \) computing a simple 2-automatic sequence:
\[
a(n) = 1,\,0,\,0,\,1,\,0,\;\dots
\]

\usetikzlibrary{automata, positioning}

\begin{figure}[H]
	\centering
	\begin{tikzpicture}[shorten >=1pt, node distance=3.2cm, on grid, auto, 
		every state/.style={draw, circle, minimum size=1.2cm, font=\normalsize}]

		\node[state, initial, accepting, label=above:{Outputs 1}] (q0) {$q_0$};
		\node[state, label=above:{Outputs 1}] (q1) [right of=q0] {$q_1$};
		\node[state, label=below:{Outputs 0}] (q2) [below of=q0] {$q_2$};
		\node[state, label=right:{Outputs 0}] (q3) [right of=q2] {$q_3$};

		\path[->]
		(q0) edge[bend left=20] node {0} (q1)
		edge[bend left=20] node {1} (q2)
		(q1) edge[bend left=20] node {0} (q0)
		edge[bend left=20] node {1} (q3)
		(q2) edge[bend left=20] node {0} (q3)
		edge[bend left=20] node {1} (q0)
		(q3) edge[loop below] node {0,1} ();
	\end{tikzpicture}
	\caption{Finite automaton computing a 2-automatic sequence \( a(n) \in \mathbb{F}_2 \). Input is the base-2 expansion of \( n \), read from least to most significant digit.}
	\label{fig:base2fsm}
\end{figure}
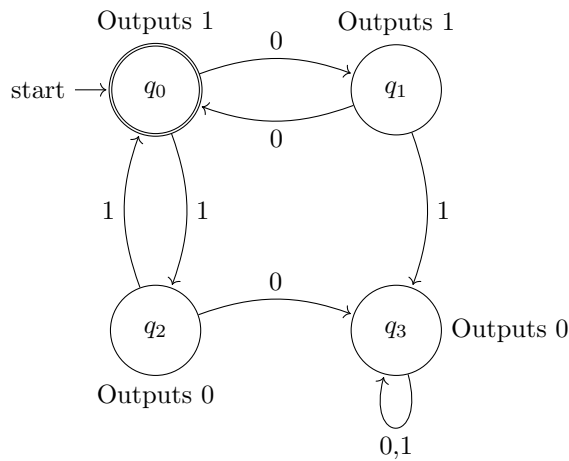

\medskip

For example:
\begin{itemize}
	\item \textbf{\( n = 0 \) (binary: \texttt{0})}: start at \( q_0 \), read \( 0 \to q_1 \), output \( 1 \).
	\item \textbf{\( n = 1 \) (binary: \texttt{1})}: start at \( q_0 \), read \( 1 \to q_2 \), output \( 0 \).
	\item \textbf{\( n = 2 \) (binary: \texttt{10})}: start at \( q_0 \), read \( 0 \to q_1 \), then \( 1 \to q_3 \), output \( 0 \).
	\item \textbf{\( n = 3 \) (binary: \texttt{11})}: start at \( q_0 \), read \( 1 \to q_2 \), then \( 1 \to q_0 \), output \( 1 \).
	\item \textbf{\( n = 4 \) (binary: \texttt{100})}: start at \( q_0 \), read \( 0 \to q_1 \), then \( 0 \to q_0 \), then \( 1 \to q_2 \), output \( 0 \).
\end{itemize}

Thus, Christol’s theorem justifies the use of automata to study constant term sequences modulo \( p \), and the size of the automaton (\( \kappa(P,p) \)) plays a central role in bounding the location of the first zero.

	\subsection{Theorem and Conjecture on the First Zero of $A_p(n)$}
	The following proposition, extracted as Proposition 6 from Section 4 of Kohen \cite{kohen}, states an upper bound on the position of the first zero in the sequence, that is, the first index \( n_0 \) such that \( A_p(n_0) = 0 \in \mathbb{F}_p \).
	\begin{proposition}[Kohen, Proposition 6, Section 4] \label{check_bound}
		Let $P$ be any Laurent polynomial and $p$ be prime. There exists a $B_{P,p} \in \mathbb{N}$, depending on $p$ and $\deg(P)$, such that if there exists some $n \in \mathbb{N}$ with $p \mid \ct{P^n}$, then there exists an $n_0 \in \mathbb{N}$ with $n_0 < B_{P,p}$ such that $p \mid \ct{P^{n_0}}$.
	\end{proposition}

	\begin{proof}
		The Rowland-Zeilberger automaton for $\ct{P^n}\bmod p$ contains at most $p^{\vert T\vert} = p^{(2m+1)^r} = p^{(2\cdot\deg(P)-1)^r}$ states. If one of the states has output $0$, then there must be a path to it from the starting state of length less than the number of states. Thus, if we let $B_{P,p} = p^{p^{(2\cdot\deg(P)-1)^r}}$, then there must be a $0$ in our sequence for some $n_0<B_{P,p}$.
	\end{proof}
	
	However, this bound of $p^{p^{(2\cdot\deg(P)-1)^r}}$ is the theoretical worst case. In Kohen's \cite{kohen} brief computer experimentation in one variable, no sequence or prime was found violating the following suggested bound:
	
	\begin{conjecture}\label{conj}
		Proposition \ref{check_bound} holds for $B_{P,p} = p^{\deg(P)}$ when $P$ is univariate.
	\end{conjecture}
	
	\ref{conj} leads to the focus of the following section where we attempt to computationally discover a polynomial $P$ which takes on a $0$ at an $n_0 > p^{deg P} = B_{P,p}$.
	
	* Note: the above proof assumes $\exists  n\in\mathbb{N} \; \text{such that} \; \text{ct}\left[P^n\right] \equiv 0 \pmod{p}$ as a premise to begin with. Experimentally it appears that most of the constant term sequences we are concerned with do have such an $n_0$ leading to $A_p(n_0) = 0 $, but the current literature does not provide a way to determine the existence of this $n_0$, which could be an exciting area of exploration. \\

	\section{Methodology}
	
	The overarching goal of the following experiment is to explore the behavior of constant term sequences modulo primes and the locations of their first zeroes. Rather than exhaustively analyzing infinitely many polynomials and their associated constant term sequences $A_p(n)$ (which unfortunately is computationally impossible), we implement a simple randomized search approach that can be repeated with various parameters and prime ranges. This reduces the search space to a manageable subset of polynomials and makes the experiment feasible.
	
	To begin, we construct a function to generate Laurent polynomials of a specified degree with randomly sampled coefficients.

	\begin{lstlisting}[language=Python, caption={Random Laurent Polynomial Generator in Sage}]
	def generate_laurent_polynomial(prime, degree):
		R.<t> = LaurentPolynomialRing(GF(prime), 1)
		coefficients = [random.randint(1, prime - 1) for _ in range(2 * degree + 1)]
		exponents = list(range(-degree, degree + 1))
		return sum(c * t^e for c, e in zip(coefficients, exponents))
	\end{lstlisting}

	 Each polynomial was randomly generated in the Laurent polynomial ring \( \mathbb{F}_p[t, t^{-1}] \) for a fixed prime \( p \). Coefficients \( a_i \in \mathbb{F}_p^\times \) (i.e., nonzero elements of the field) were selected uniformly at random from the set \( \{1, 2, \dots, p-1\} \). Each generated polynomial had the form
	 \[
	 P(t) = \sum_{i = -d}^{d} a_i t^i,
	 \]
	 where \( d \) is a fixed degree bound, and all \( 2d + 1 \) coefficients \( a_i \) are nonzero.

	Using this function, we developed a simple program to perform preliminary experiments and investigate trends. For each randomly generated polynomial \( P \), we used Kohen’s \texttt{LinRep} module to compute the first index \( n_0 \) at which the constant term sequence \( A_p(n) \) satisfies \( A_p(n_0) = 0 \) over a range of primes \( p \). If such an index exists, it is referred to as the \emph{shortest zero} and recorded. If no zero occurs (i.e., the associated automaton has no zero-outputting state), the corresponding prime \( p \) is noted separately.
	
	For each polynomial, we plotted the shortest zero values against the modulus \( p \), using blue markers for observed shortest zeros and red stars for primes with no zero. These plots allowed visual inspection of trends. While initial plots appeared largely unstructured, further examination revealed that for some polynomials, the shortest zero values grew approximately linearly or quadratically with \( p \). However, such patterns were not consistent across all examples, preventing any broad generalization.
	
	\vspace{1em}
	The figure below presents a "quilt" of plots for 25 randomly generated Laurent polynomials of degree at most 4:

	\vspace*{-1em}
	\begin{figure}[H]
		\centering
		\includegraphics[width=\textwidth]{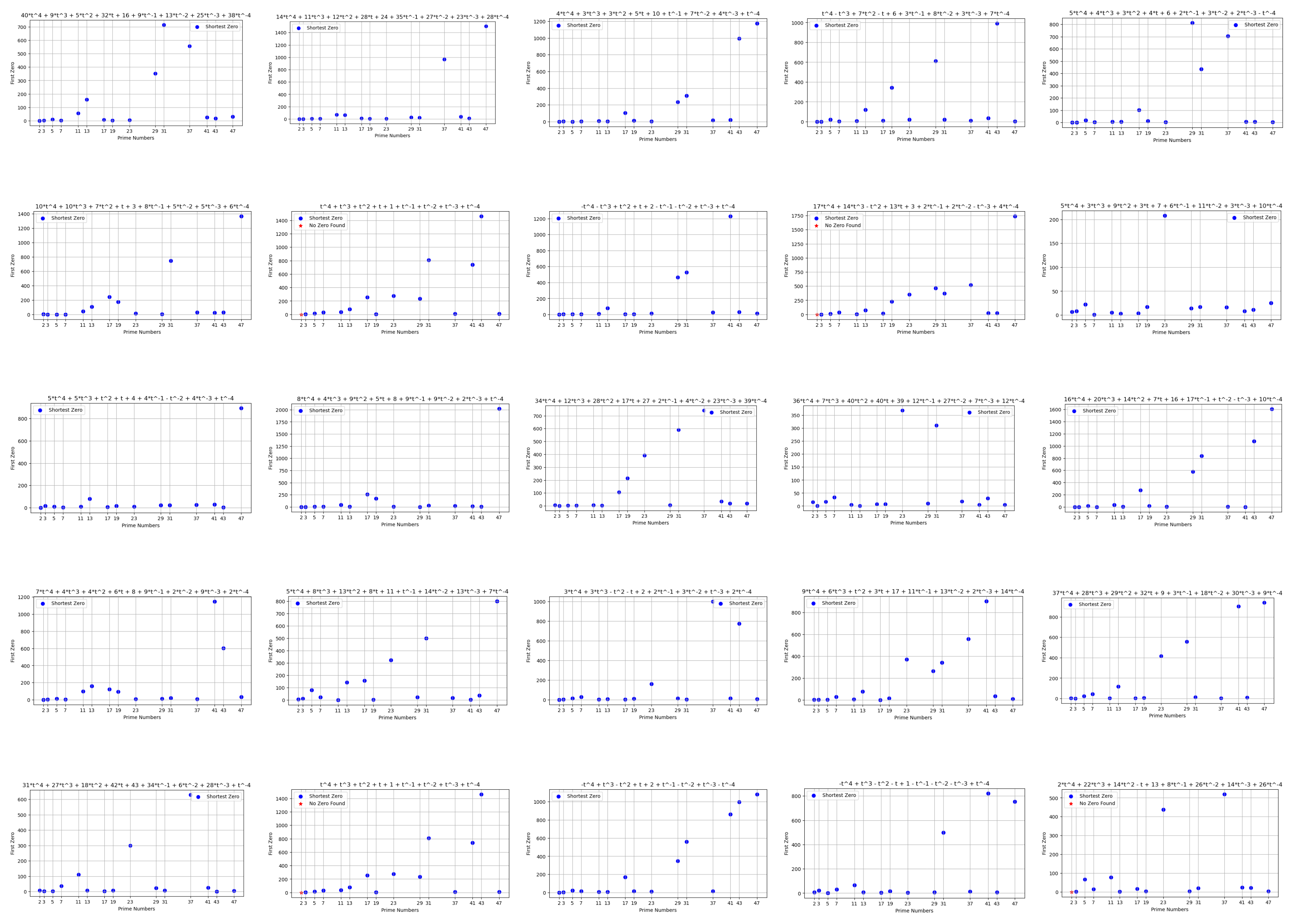}
		\caption{\small{Shortest zero values across primes for various randomly generated Laurent polynomials. Each subplot represents a different polynomial, with prime numbers on the x-axis and the first zero index on the y-axis.}}
		\label{fig:quilt}
	\end{figure}

	During extensive experimentation, we observed that degree-2 univariate Laurent polynomials often exhibited erratic shortest-zero behavior. This made them an excellent search space for potential counterexamples to Proposition~\ref{check_bound}, particularly because bound violations tended to occur at small primes, allowing fast and inexpensive computation across many polynomials.

 	\clearpage
 	
	To systematically explore this search space, we implemented the following experimental logic:
	\begin{lstlisting}[language=Python, caption=Main Loop Logic]
	for _ in range(num_polynomials):
		primes_list = list(primerange(2, max_prime_mod))
		prime = random.choice(primes_list)
		poly, degree = generate_laurent_polynomial(prime, degree)
		
		shortest_zeros = []
		none_primes = []
		valid_primes = []
		violations = []
		violation_primes = []
		
		for p in primes_list:
			try:
				R.<t> = LaurentPolynomialRing(GF(p), 1)
				P = R(poly)
				Q = R(1)
				sz = LinRep.compute_shortest_zero(P, Q, p, 50000)
				
				if sz is not None:
					if sz > (p ** degree):
						violations.append(sz)
						violation_primes.append(p)
					else:
						valid_primes.append(p)
						shortest_zeros.append(sz)
				else:
					none_primes.append(p)
			
			except Exception as e:
				print(f"Error for prime {p}: {e}")
	\end{lstlisting}

	In order to search for examples that break $B_{P,p} = p^{\deg(P)}$, the main loop takes a polynomial's constant term sequence $ct [ P(t)^n] \mod p$ and computes its shortest zero, \texttt{sz}, for each prime $p$ up to a user-defined bound (i.e compute the shortest zero for $0 \leq p \leq 7$). In each computation, the main loop checks if $A_p(n_0) = 0$ when $n_0 > p^{\deg P}$ . If it is, it is considered a \texttt{violation}, otherwise, it is considered a \texttt{valid\_prime}. Doing this for each generated polynomial leads to plots of the first zero positions in the associated constant term sequences.

	\section{Results}
	
	In this section, we present an explicit counterexample to conjecture \ref{conj}, which posits that for any univariate Laurent polynomial \( P \in \mathbb{F}_p[t, t^{-1}] \) of degree \( d \), the first zero of the constant term sequence \( A_p(n) = \ct{P^n} \bmod p \) occurs at some \( n_0 < p^{\deg(P)} \).
	
	To find such an example, we utilized the program logic outlined in section 3 to conduct a randomized search over degree-2 Laurent polynomials until we found a violating polynomial.
	
	\subsection{An Explicit Counterexample to Conjecture~\ref{conj}}
	
	Let \( P(t) = 4t^2 + 6t + 1 + 6t^{-1} + 0t^{-2} \in \mathbb{F}_7[t, t^{-1}] \) be a univariate Laurent polynomial of degree 2 over \( \mathbb{F}_7 \), obtained by reducing coefficients of \( 32t^2 + 13t + 1 + 27t^{-1} + 35t^{-2} \) modulo 7. 
	Computing the associated constant term sequence \( A_7(n) = \operatorname{ct}(P(t)^n) \bmod 7 \), we find that the first zero occurs at \( n_0 = 225 \), whereas Conjecture~\ref{conj} predicts it must occur at some \( n_0 < 7^2 = 49 \).
	Since \( 225 > 49 \), this provides a counterexample to Conjecture~\ref{conj}.
	
	\vspace{1em}
	
	\begin{figure}[H]
		\centering
		\includegraphics[width=0.5\textwidth]{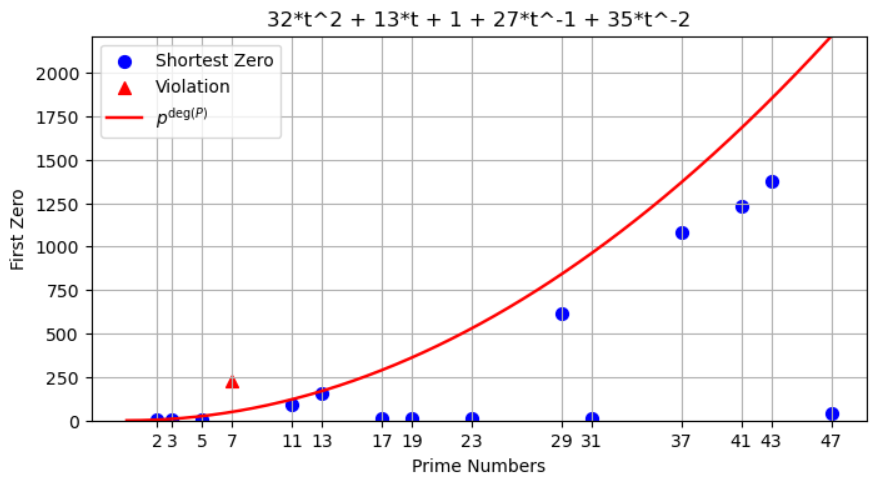}
		\caption{Shortest zero values across primes for \( P(t) = 32t^2 + 13t + 1 + 27t^{-1} + 35t^{-2} \) with violations marked.}
		\label{fig:counterexample}
	\end{figure}

	\subsection{Ten Additional Counterexamples to Conjecture~\ref{conj}}
	
	\noindent
	 Reduced polynomials over \(\mathbb{F}_p\) yielding counterexamples to Conjecture~\ref{conj}.

	\begin{align*}
		P_1(t) &= t^2 + 2t + 3 + 10t^{-1} + 2t^{-2} \in \mathbb{F}_{11} \\
		P_2(t) &= 3t^2 + 0t + 2 + 3t^{-1} + 3t^{-2} \in \mathbb{F}_5 \\
		P_3(t) &= t^2 + t + 1 + t^{-1} + 2t^{-2} \in \mathbb{F}_3 \\
		P_4(t) &= 4t^2 + 0t + 4 + 4t^{-1} + t^{-2} \in \mathbb{F}_5 \\
		P_5(t) &= 0t^2 + 3t + 3 + 6t^{-1} + 1t^{-2} \in \mathbb{F}_7 \\
		P_6(t) &= t^2 + 0t + 4 + 4t^{-1} + 9t^{-2} \in \mathbb{F}_{11} \\
		P_7(t) &= 0t^2 + 3t + 1 + 0t^{-1} + 4t^{-2} \in \mathbb{F}_5 \\
		P_8(t) &= 2t^2 + 2t + 1 + 2t^{-1} + t^{-2} \in \mathbb{F}_3 \\
		P_9(t) &= 4t^2 + t + 2 + 4t^{-1} + t^{-2} \in \mathbb{F}_5 \\
		P_{10}(t) &= t^2 + 2t + 1 + 2t^{-1} + 2t^{-2} \in \mathbb{F}_3
	\end{align*}

	\section{Automatic Bound on First Zero}
	
	\begin{proposition} \label{auto_bound}
		Let \( P(t) \in \mathbb{F}_p[t, t^{-1}] \) be a univariate Laurent polynomial, and define the sequence \( A_p(n) = \ct{P(t)^n} \bmod p \). Let \( \kappa(P, p) \in \mathbb{N} \) denote the number of states in a minimal deterministic finite automaton that computes \( A_p(n) \) from the base-\( p \) representation of \( n \). Then, if the sequence \( A_p(n) \) ever takes the value \( 0 \), the index \( n_0 \) of the first zero satisfies
		\[
		p^{\kappa(P, p)-1} \leq n_0 < p^{\kappa(P, p)}.
		\]
	\end{proposition}
	
	\begin{proof}
		We proceed by cases on the number of digits of \( n \) in base-\( p \).
		
		\medskip
		
		\noindent \textbf{Case 1:} \( n < p^{\kappa(P,p)-1} \)
		
		\noindent
		The base-\( p \) representation of \( n \) has fewer than \( \kappa(P,p) \) digits. Since the automaton has \( \kappa(P,p) \) states and transitions depend on reading exactly \( \kappa(P,p) \) digits, inputs with fewer digits may not explore the full set of states. In particular, the zero-outputting state might not be reachable from short inputs. Thus, no guarantee exists that a zero will occur in this range.
		
		\medskip
		
		\noindent \textbf{Case 2:} \( p^{\kappa(P,p)-1} \leq n < p^{\kappa(P,p)} \)
		
		\noindent
		The base-\( p \) representation of \( n \) has exactly \( \kappa(P,p) \) digits. In this range, the automaton is capable of exploring all reachable states, because the input length matches the full depth of the automaton. If the automaton has a state that outputs \( 0 \), there must exist an input of this length that reaches it. Therefore, the first zero must occur somewhere in this interval.
		
		\medskip
		
		\noindent
		Combining both cases, the first zero \( n_0 \) must satisfy
		\[
		p^{\kappa(P,p)-1} \leq n_0 < p^{\kappa(P,p)}.
		\]
	\end{proof}

	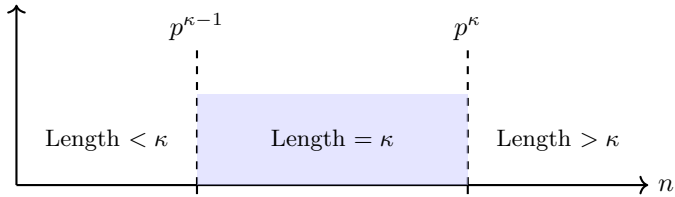
\begin{figure}[H]
		\centering
		\begin{tikzpicture}[scale=1.2, thick]
			\draw[->] (0,0) -- (7,0) node[right] {$n$};
			\draw[->] (0,0) -- (0,2) node[above] {};
			
			\draw[dashed] (2,0) -- (2,1.5) node[above] {$p^{\kappa-1}$};
			\draw[dashed] (5,0) -- (5,1.5) node[above] {$p^{\kappa}$};
			
			\fill[blue!20,opacity=0.5] (2,0) rectangle (5,1);
			
			\node at (1,0.5) {\small Length $< \kappa$};
			\node at (3.5,0.5) {\small Length $= \kappa$};
			\node at (6,0.5) {\small Length $> \kappa$};
			
			\foreach \x in {2,5}
			\draw (\x,0) -- (\x,-0.1);
			
		\end{tikzpicture}
		\caption{Range where the shortest zero must occur: inputs with base-\( p \) length exactly \(\kappa\) digits lie between \( p^{\kappa-1} \) and \( p^\kappa \).}
		\label{fig:kappa_bounds}
	\end{figure}

	\medskip
	
	\noindent
	\textbf{Remark.} In practice, the state set \( \kappa(P, p) \) can be computed explicitly using with the \texttt{ConstantTermSequences} Sage library \cite{kohen}. This usually tends to be dramatically lower than the theoretical worst-case bound.

	\section{Conclusion}
	
	Through computational experimentation, we have provided explicit counterexamples showing that the conjectured bound, \(p^{\deg(P)} \), does not universally hold for univariate Laurent polynomials. Our results demonstrate the necessity of revising or supplementing the original conjecture to account for additional structural factors. 
	
	We also established an automatic bound based on the number of states \(\kappa(P,p)\) in the minimal finite automaton computing the sequence \(A_p(n)\), namely:
	\[
	p^{\kappa(P,p)-1} \leq n_0 < p^{\kappa(P,p)}.
	\]
	This bound provides a more reliable and computable constraint on the first occurrence of zero in constant term sequences modulo primes. Future work will aim to further characterize the algebraic or combinatorial structures that govern the distribution of zeros in these sequences.

	\section*{Acknowledgments}
	I would like to thank Nadav Kohen for introducing me to the subject of his thesis and for his guidance during the semester for this project. It has been an absolute pleasure to learn so much. I'd also like to extend thanks to those who run the Undergraduate Research Lab, LEMMA, at Indiana University, for the opportunity get exposure to mathematical research this early in my academic career.
	
\bibliographystyle{plain}

\end{document}